\documentclass{amsart}
\usepackage{amssymb,amsmath,latexsym}
\usepackage{amsthm}
\usepackage{fontenc}
\usepackage{amssymb}
\numberwithin{equation}{section}

\newtheorem{theorem}{Theorem}[section]

\newtheorem{lemma}[theorem]{Lemma}

\begin{document}
\author{Alexander E Patkowski}
\title{Partitions related to positive definite binary quadratic forms}

\maketitle
\begin{abstract} The purpose of this paper is to present a collection of interesting generating functions for partitions which have connections to positive definite binary quadratic forms. In establishing our results we obtain some new Bailey pairs.  \end{abstract}

% AMS keywords (used in AMS journals)

% AMS subject classifications (used in AMS journals)

\section{Introduction and main theorems}

In a study on lacunary partition functions [7], Lovejoy offered a collection of interesting partition functions which satisfy an estimate given by P. Bernays. Therein, Theorem 1 is constructed using a special Bailey pair which connects $q$-series to positive definite binary quadratic forms. There appears to be few studies in the literature developing connections between positive definite quadratic forms and partitions through these types of Bailey pairs. By positive definite, we take the usual definition where we write $Q(x,y)=ax^2+bxy+cy^2,$ with $a>0,$ $b^2-4ac<0,$ for $a,b,c\in\mathbb{Z}.$ We offer a collection of partitions that we believe to be new and interesting, and follow from new Bailey pairs that are similar to the one offered in [7].
\par
We will be applying $q$-series notation that is widely used throughout the literature [6]. We shall put $(z;q)_n=(z)_n:=(1-z)(1-zq)\cdots(1-zq^{n-1}).$ It is taken that $q\in\mathbb{C},$ and all of our series converge in the unit circle, $0<|q|<1.$
\par In our first example, we consider a partition function that is related to the $f_1(q)$ studied in [4], but with a different weight function.
\begin{theorem}\label{thm:thm1} Let $P_{m,j}(n)$ denote the number of partitions of $n$ into $m$ distinct parts and one part $2m+1$ that may repeat any number of times or not appear at all. Here $j$ is the largest distinct part. Then,
$$\sum_{m,n,j\ge0}(-1)^jP_{m,j}(n)q^n= \frac{1}{2}\sum_{n\ge0}q^{n^2+n/2}(1+q^{n+1/2})\sum_{|j|\le n}q^{j^2/2}$$
$$+\frac{1}{2}\sum_{n\ge0}(-1)^nq^{n^2+n/2}(1-q^{n+1/2})\sum_{|j|\le n}(-1)^jq^{j^2/2}.$$
\end{theorem}
We recall that an overpartition is a partition of $n$ where the first occurrence of a number may be overlined [9].
\begin{theorem}\label{thm:thm2} Let $Q_{m,j}(n)$ denote the number of partition pairs $(\mu,\lambda)$ of $n$ where $\mu$ is the number of overpartitions into even parts $\le2m$ with $j$ equal to the number of parts, and $\lambda$ is the number of partitions into odd parts $\le2m+1$ where (i) all odd numbers $< 2m+1$ appear as a part and an even number of times. (ii) the part $2m+1$ may repeat any number of times or not appear at all.
$$\sum_{m,n,j\ge0}(-1)^jQ_{m,j}(n)q^n= \frac{1}{2}\sum_{n\ge0}q^{3n^2/2}(1+q^{2n+1})(1+q^{n+1/2})\sum_{|j|\le n}q^{j^2/2}$$
$$+\frac{1}{2}\sum_{n\ge0}(-1)^nq^{3n^2/2}(1+q^{2n+1})(1-q^{n+1/2})\sum_{|j|\le n}(-1)^jq^{j^2/2}.$$
\end{theorem}

\begin{theorem}\label{thm:thm3} Let $R_{m,j}(n)$ denote the number of partition pairs $(\mu,\pi)$ of $n$ where $\mu$ is the number of overpartitions into even parts $\le2m$ with $j$ equal to the number of parts, and $\lambda$ is the number of partitions into odd parts $\le2m+1$ where (i) all odd numbers $< 2m+1$ appear as a part and at least once. (ii) the part $2m+1$ may repeat any number of times or not appear at all. Then,
$$\sum_{m,j,n\ge0}(-1)^{j}R_{m,j}(n)q^n= \frac{1}{2}\sum_{n\ge0}q^{n^2/2}(1+q^{2n+1})(1+q^{n+1/2})\sum_{|j|\le n}q^{j^2/2}$$
$$+\frac{1}{2}\sum_{n\ge0}(-1)^nq^{n^2/2}(1+q^{2n+1})(1-q^{n+1/2})\sum_{|j|\le n}(-1)^jq^{j^2/2}.$$
\end{theorem}

\section{Proof of Theorems}
Here we give the proofs of our theorems, which will require some lemmas from the literature and also some new Bailey pairs. First we note that a pair $(\alpha_n(a; q),\beta_n(a,q))$ is said to be a Bailey pair [2, 12] with respect to
$(a,q)$ if
\begin{equation}\beta_n(a,q)=\sum_{0\le j\le n}\frac{\alpha_j(a,q)}{(q;q)_{n-j}(aq;q)_{n+j}}.\end{equation} It is known that [12]
\begin{equation}\sum_{n\ge0}(X_1)_n(X_2)_n(aq/X_1X_2)^n\beta_n(a,q)=\frac{(aq/X_1)_{\infty}(aq/X_2)_{\infty}}{(aq)_{\infty}(aq/X_1X_2)_{\infty}}\sum_{n\ge0}\frac{(X_1)_n(X_2)_n(aq/X_1X_2)^n\alpha_n(a,q)}{(aq/X_1)_n(aq/X_2)_n}.\end{equation}
We need to write down a result that was established by Lovejoy [8,eq.(2.4)--(2.5)]
\begin{lemma} (Lovejoy [8, eq.(2.4)--(2.5)]) If $(\alpha_n(a; q),\beta_n(a,q))$ is a Bailey pair, then so is $(\alpha^{*}_n(aq,b, q),\beta^{*}_n(aq,b,q))$ where
\begin{equation}\alpha^{*}_n(aq,b, q)=\frac{(1-aq^{2n+1})(aq/b;q)_n(-b)^nq^{n(n-1)/2}}{(1-aq)(bq)_n}\sum_{n\ge j\ge0}\frac{(b)_j}{(aq/b)_j}(-b)^{-j}q^{-j(j-1)/2}\alpha_j(a,q)\end{equation}
\begin{equation}\beta^{*}_n(aq,b,q)=\frac{(1-b)}{1-bq^n}\beta_n(a,q).\end{equation}
\end{lemma}
If we let $a=1$ in Lemma 2.1, divide both sides by $(1-b),$ and add the resulting Bailey pair to itself after replacing $b$ by $-b$ we obtain the next Bailey pair.

\begin{lemma} If $(\alpha_n(1, q),\beta_n(1,q))$ is a Bailey pair, then so is $(L_{1(n)}(q,b, q),L_{2(n)}(q,b,q))$ where
\begin{equation}L_{1(n)}(q,b,q)=\frac{1}{2(1-b)}\alpha^{*}_n(q,b, q) +\frac{1}{2(1+b)}\alpha^{*}_n(q,-b, q),\end{equation}
\begin{equation}L_{2(n)}(q,b,q)=\frac{\beta_n(1,q)}{1-b^2q^{2n}}.\end{equation}
\end{lemma}
We need a result that was written down in Patkowski [10].
\begin{lemma} If $(\alpha(a,q),\beta(a,q))$
forms a Bailey pair with respect to $(a, q),$ then $(\alpha'_n(a^2,q^2),\beta'_n(a^2,q^2))$ forms a Bailey pair with respect to $(a^2,q^2)$, if
\begin{equation}\alpha'_n(a^2,q^2)=\frac{(1+aq^{2n})}{(1+a)q^n}\alpha_n(a,q),\end{equation}
\begin{equation}\beta'_n(a^2,q^2)=\frac{q^{-n}}{(-a;q)_{2n}}\sum_{n\ge j \ge0}\frac{(-1)^{n-j}q^{(n-j)^2-(n-j)}}{(q^2;q^2)_{n-j}}\beta_j(a,q).\end{equation}
\end{lemma}
We take the E(1) Bailey pair relative to $a=1$ from Slater's list [11], $\alpha_{0}(1,q)=1,$
\begin{equation}\alpha_n(1,q)=2(-1)^nq^{n^2},\end{equation} when $n>0,$ and
\begin{equation}\beta_n(1,q)=\frac{1}{(q^2;q^2)_n},\end{equation}
and insert it into Lemma 2.2. We then insert the resulting pair into Lemma 2.3 to obtain our main Bailey pair. To obtain the following lemma, we require use of an identity found in Fine's text [pg.26, eq.(20.41): $a=b/q,$ $c=0.$]
$$(t)_{\infty}\sum_{n\ge0}\frac{t^n}{(q)_n(1-bq^{n})}=\sum_{n\ge0}\frac{(-t)^nb^nq^{n(n-1)/2}}{(b)_{n+1}}.$$

\begin{lemma} Define
$$U_{n}(q,b, q):=\frac{(1-q^{4n+2})(q/b;q)_n(-b)^nq^{n(n-1)/2}}{q^n(1-b)(1-q^2)(bq)_n}\left(1+2\sum_{n\ge j>0}\frac{(b)_j}{(q/b)_j}(b)^{-j}q^{j(j+1)/2}\right).$$
Then, $(U_{1(n)}(q^2,b,q^2), U_{2(n)}(q^2,b,q^2))$ form a Bailey pair where $U_{1(n)}(q^2,b,q^2)=\frac{1}{2}U_{n}(q,b, q)+\frac{1}{2}U_{n}(q,-b, q),$ and
$$U_{2(n)}(q^2,b,q^2)=\frac{(-b^2)^nq^{n(n-2)}}{(-q)_{2n}(b^2;q^2)_{n+1}}.$$
\end{lemma}
\begin{proof}[Proof of Theorem~\ref{thm:thm1}] We take take the Bailey pair that results from inserting (2.9)--(2.10) into Lemma 2.2 with $b=q^{1/2}.$ Then we apply (2.2) with $X_1=q$ and let $X_2\rightarrow\infty$ to obtain
\begin{equation}\sum_{n\ge0}\frac{(-1)^nq^{n(n+1)/2}}{(-q)_n(1-q^{2n+1})}= \frac{1}{2}\sum_{n\ge0}q^{n^2+n/2}(1+q^{n+1/2})\sum_{|j|\le n}q^{j^2/2}\end{equation}
$$+\frac{1}{2}\sum_{n\ge0}(-1)^nq^{n^2+n/2}(1-q^{n+1/2})\sum_{|j|\le n}(-1)^jq^{j^2/2}.$$
It is clear to see that $(-1)^mq^{m(m+1)/2}/(-q)_m$ generates partitions of $n$ into $m$ distinct parts weighted by $-1$ raised to the largest part. The component $(1-q^{2m+1})^{-1}$ generates a partition of the part $2m+1$ which may repeat or not appear at all. Combining the two gives the partition described in the theorem.
\end{proof}

\begin{proof}[Proof of Theorem~\ref{thm:thm2}] We take the Bailey pair in Lemma 2.4 with $b=q^{1/2}$ and then insert it into (2.2) with $X_1=q^2,$ $X_2\rightarrow\infty,$ to get 
\begin{equation}\sum_{n\ge0}\frac{(q^2;q^2)_nq^{2n^2}}{(-q^2;q^2)_n(q^2;q^4)_n(1-q^{2n+1})}= \frac{1}{2}\sum_{n\ge0}q^{3n^2/2}(1+q^{2n+1})(1+q^{n+1/2})\sum_{|j|\le n}q^{j^2/2}\end{equation}
$$+\frac{1}{2}\sum_{n\ge0}(-1)^nq^{3n^2/2}(1+q^{2n+1})(1-q^{n+1/2})\sum_{|j|\le n}(-1)^jq^{j^2/2}.$$
Now $(q^2;q^2)_m/(-q^2;q^2)_m$ generates an overpartition of $n$ into even parts $\le2m$ weighted by $-1$ raised to the number of parts. The function $q^{m^2}/(q^2;q^4)_m(1-q^{2m+1})$ generates the partition $\lambda$ given in the theorem. To see this, we write
$$\frac{q^{2n^2}}{(q^2;q^4)_n(1-q^{2n+1})}=\frac{q^{1+1+3+3+\cdots (2n-1)+(2n-1)}}{(1-q^{1+1})(1-q^{3+3})\cdots(1-q^{2n-1+2n-1})(1-q^{2n+1})}. $$

\end{proof}

\begin{proof}[Proof of Theorem~\ref{thm:thm3}] We take the Bailey pair in Lemma 2.4 with $b=q^{1/2}$ and then insert it into (2.2) with $X_1=q^2,$ $X_2=-q,$ and rewrite to get 
\begin{equation}\sum_{n\ge0}\frac{(q^2;q^2)_nq^{n^2}}{(-q^2;q^2)_n(q;q^2)_{n+1}}= \frac{1}{2}\sum_{n\ge0}q^{n^2/2}(1+q^{2n+1})(1+q^{n+1/2})\sum_{|j|\le n}q^{j^2/2}\end{equation}
$$+\frac{1}{2}\sum_{n\ge0}(-1)^nq^{n^2/2}(1+q^{2n+1})(1-q^{n+1/2})\sum_{|j|\le n}(-1)^jq^{j^2/2}.$$
The partition generating function on the left side is quite similar to our previous theorem, and so the remaining details are left to the reader.

\end{proof}
\rm
\par
Here we can observe that more partition functions may be obtained by instead selecting different Bailey pairs from Slater's list [11] in conjunction with Lemma 2.1. This principal idea is of course aided with the application of the work [3] to obtain simple forms of $\beta_n(a,q)$ in the same way we have done here with Lemma 2.3.

1390 Bumps River Rd. \\*
Centerville, MA
02632 \\*
USA \\*
E-mail: alexpatk@hotmail.com
\end{document}